\documentclass[a4paper,10pt]{amsart}
\usepackage{amsmath,amsthm,amssymb,latexsym,enumerate,xcolor,hyperref}
\usepackage{graphicx}

\usepackage{comment}

\numberwithin{equation}{section}

\begin{document}

	\newtheorem{thm}{Theorem}[section]
	\newtheorem{prop}[thm]{Proposition}
	\newtheorem{lem}[thm]{Lemma}
	\newtheorem{cor}[thm]{Corollary}
	\newtheorem{rem}[thm]{Remark}
	\newtheorem*{defn}{Definition}

	\newtheorem{definit}[thm]{Definition}
	\newtheorem{setting}{Setting}
	\renewcommand{\thesetting}{\Alph{setting}}
	
	\newcommand{\DD}{\mathbb{D}}
	\newcommand{\NN}{\mathbb{N}}
	\newcommand{\ZZ}{\mathbb{Z}}
	\newcommand{\QQ}{\mathbb{Q}}
	\newcommand{\RR}{\mathbb{R}}
	\newcommand{\CC}{\mathbb{C}}
	\renewcommand{\SS}{\mathbb{S}}

	\renewcommand{\theequation}{\arabic{section}.\arabic{equation}}

	\newcommand{\supp}{\mathop{\mathrm{supp}}}    
	
	\newcommand{\re}{\mathop{\mathrm{Re}}}   
	\newcommand{\im}{\mathop{\mathrm{Im}}}   
	\newcommand{\dist}{\mathop{\mathrm{dist}}}  
	\newcommand{\link}{\mathop{\circ\kern-.35em -}}
	\newcommand{\spn}{\mathop{\mathrm{span}}}   
	\newcommand{\ind}{\mathop{\mathrm{ind}}}   
	\newcommand{\rank}{\mathop{\mathrm{rank}}}   
	\newcommand{\Fix}{\mathop{\mathrm{Fix}}}   
	\newcommand{\codim}{\mathop{\mathrm{codim}}}   
	\newcommand{\conv}{\mathop{\mathrm{conv}}}   
	\newcommand{\epsi}{\mbox{$\varepsilon$}}
	\newcommand{\eps}{\mathchoice{\epsi}{\epsi}
		{\mbox{\scriptsize\epsi}}{\mbox{\tiny\epsi}}}
	\newcommand{\cl}{\overline}
	\newcommand{\pa}{\partial}
	\newcommand{\ve}{\varepsilon}
	\newcommand{\zi}{\zeta}
	\newcommand{\Si}{\Sigma}
	\newcommand{\cA}{{\mathcal A}}
	\newcommand{\cG}{{\mathcal G}}
	\newcommand{\cH}{{\mathcal H}}
	\newcommand{\cI}{{\mathcal I}}
	\newcommand{\cJ}{{\mathcal J}}
	\newcommand{\cK}{{\mathcal K}}
	\newcommand{\cL}{{\mathcal L}}
	\newcommand{\cN}{{\mathcal N}}
	\newcommand{\cR}{{\mathcal R}}
	\newcommand{\cS}{{\mathcal S}}
	\newcommand{\cT}{{\mathcal T}}
	\newcommand{\cU}{{\mathcal U}}
	\newcommand{\OM}{\Omega}
	\newcommand{\B}{\bullet}
	\newcommand{\ol}{\overline}
	\newcommand{\ul}{\underline}
	\newcommand{\vp}{\varphi}
	\newcommand{\AC}{\mathop{\mathrm{AC}}}   
	\newcommand{\Lip}{\mathop{\mathrm{Lip}}}   
	\newcommand{\es}{\mathop{\mathrm{esssup}}}   
	\newcommand{\les}{\mathop{\mathrm{les}}}   
	\newcommand{\nid}{\noindent}
	\newcommand{\pzr}{\phi^0_R}
	\newcommand{\pir}{\phi^\infty_R}
	\newcommand{\psr}{\phi^*_R}
	\newcommand{\pow}{\frac{N}{N-1}}
	\newcommand{\ncl}{\mathop{\mathrm{nc-lim}}}   
	\newcommand{\nvl}{\mathop{\mathrm{nv-lim}}}  
	\newcommand{\la}{\lambda}
	\newcommand{\La}{\Lambda}    
	\newcommand{\de}{\delta}    
	\newcommand{\fhi}{\varphi} 
	\newcommand{\ga}{\gamma}    
	\newcommand{\ka}{\kappa}   
	
	\newcommand{\core}{\heartsuit}
	\newcommand{\diam}{\mathrm{diam}}

	\newcommand{\lan}{\langle}
	\newcommand{\ran}{\rangle}
	\newcommand{\tr}{\mathop{\mathrm{tr}}}
	\newcommand{\diag}{\mathop{\mathrm{diag}}}
	\newcommand{\dv}{\mathop{\mathrm{div}}}
	
	\newcommand{\al}{\alpha}
	\newcommand{\be}{\beta}
	\newcommand{\Om}{\Omega}
	\newcommand{\na}{\nabla}
	
	\newcommand{\cC}{\mathcal{C}}
	\newcommand{\cM}{\mathcal{M}}
	\newcommand{\nr}{\Vert}
	\newcommand{\De}{\Delta}
	\newcommand{\cX}{\mathcal{X}}
	\newcommand{\cP}{\mathcal{P}}
	\newcommand{\om}{\omega}
	\newcommand{\si}{\sigma}
	\newcommand{\te}{\theta}
	\newcommand{\Ga}{\Gamma}
	
	\newcommand{\vV}{\mathbf{v}}
	\newcommand{\lbunu}{\ul{m}}
	\newcommand{\ca}{\tilde{a}}
	\newcommand{\Vve}{\ul{\varepsilon}}

	\newcommand{\cF}{\mathcal{F}}
	\newcommand{\cD}{\mathcal{D}}
	
	\title[Remarks about the mean value property and
	Poincar\'e
	inequalities]{Remarks about the mean value property
		and
		\\
		some weighted Poincar\'e-type inequalities
		}

	\author{Giorgio Poggesi}
	\address{Department of Mathematics and Statistics, The University of Western Australia, 35 Stirling Highway, Crawley, Perth, WA 6009, Australia}
	%
	%
	\email{giorgio.poggesi@uwa.edu.au}
	
	\begin{abstract}
	We start providing a quantitative stability theorem for the rigidity
	of
	an overdetermined problem involving harmonic functions in a punctured domain. Our approach is inspired by and based on the proof of rigidity established by Enciso and Peralta-Salas in \cite{EPS}, and reveals essential differences with respect to the stability results obtained in the literature for the classical overdetermined Serrin problem. 	
		
	Secondly, we provide new weighted Poincar\'e-type inequalities for vector fields. These are crucial tools for the study of the quantitative stability issue initiated
	in \cite{Pog3} concerning a class of rigidity results involving mixed boundary value problems.
	%
	%
	
	
	Finally, we provide a mean value-type property and an associated weighted Poincar\'e-type inequality for harmonic functions in cones.
	A duality relation between this new mean value property and a partially overdetermined boundary value problem is discussed, providing an extension of a classical result obtained in \cite{PS}. 
	\end{abstract}

\maketitle

\section{Introduction}

The paper comprises three self-contained sections. 

Section \ref{sec:Quantitative rigidity delta} provides quantitative stability (Theorem \ref{thm:stability delta}) of a rigidity result for an overdetermined problem that is related to the spherical mean value property for harmonic functions. Such a stability result reveals essential differences with respect to the stability results obtained so far in the literature for the classical overdetermined Serrin problem: details are provided in Section \ref{sec:Quantitative rigidity delta}.
The proof of our stability result is inspired by and based on the proof of rigidity established by Enciso and Perlata-Salas in \cite{EPS}. 

In Section \ref{sec:Poincare}, we provide new weighted Poincar\'e-type inequalities for vector fields. These are crucial tools for the study of the quantitative stability issue initiated
in \cite{Pog3} concerning a class of rigidity results involving mixed boundary value problems.

Section \ref{sec:cones} provides a new mean value-type property and an associated weighted Poincar\'e-type inequality for harmonic functions in cones. Moreover, a duality relation between this new mean value property and a partially overdetermined boundary value problem is discussed, providing an extension of a classical result obtained in \cite{PS}. 

\section{Quantitative stability for an overdetermined problem}\label{sec:Quantitative rigidity delta}

The present section provides quantitative stability of a rigidity result for an overdetermined problem involving harmonic functions in a punctured domain.
More precisely, we consider the following Dirichlet problem:
\begin{equation}\label{eq:delta}
	\begin{cases}
		- \De u = | \pa \Om| \, \de_0 \quad & \mbox{ in } \Om , 
		\\
		u= c \quad & \mbox{ on } \pa\Om,
	\end{cases}
\end{equation} 
where $\de_0$ denotes the Dirac delta centered at the origin $0 \in \Om \subset \RR^N $, and $c$ is a constant.

Rigidity for Problem \eqref{eq:delta} under the overdetermined condition $|\na u| \equiv const.$ on $\pa \Om$ was proved in \cite[Theorems III.1-III.2]{PS} via duality, in \cite{AR} via the method of moving planes, and in \cite{EPS} with a proof in the wake of Weinberger (\cite{We}).
	In the planar case ($N=2$), rigidity and stability were obtained in \cite{AM1,AM2}, by means of the theory of conformal mappings (see (iii) of Remark \ref{rem:generalizations and comparison with planar case} for details).
Such an overdetermined problem is related to the spherical mean value property for harmonic functions via the duality theorem established in \cite[Theorems III.1]{PS}.

To the author's knowledge, for $N>2$ the stability issue for such an overdetermined problem has not been studied in the literature yet. On the contrary, stability for the classical overdetermined Serrin problem (\cite{Se,We,NT}) has been extensively studied by several authors.
The stability results for the classical Serrin problem so far obtained in the literature can be essentially divided into two groups\footnote{See also \cite{GO,O} for a further recent alternative approach based on a modified implicit function theorem.}, depending on the method employed. The first -- which includes \cite{ABR, CMV} -- follows the tracks of \cite{Se} and is based on a quantitative study of the method of moving planes, the second --  which includes \cite{BNST, Feldman, MP2, MP3} --  follows the tracks of \cite{We} and is based on integral identities and inequalities.
A common feature of the second group is that quantitative stability results are achieved based on proofs of rigidity which follow the tracks of 
%
%
\cite{We}, but, in contrast with \cite{We}, avoid using the maximum principle for the $P$-function\footnote{More precisely, \cite{BNST, Feldman} are based on the proof of rigidity established in \cite{BNST ARMA}, whereas \cite{MP2,MP3} are based on the proof
	in \cite[Theorem 2.1]{MP2} which was inspired by \cite[Theorems I.1, I.2]{PS}. We refer the interested reader to the surveys \cite{NT, Ma} for more details.}.
On the contrary, the following proof, which is inspired by the proof of rigidity in \cite{EPS}, shows that, for Problem \eqref{eq:delta}, the maximum principle for the $P$-function and the isoperimetric inequality nicely combine to allow to achieve quantitative stability estimates.

\begin{rem}\label{rem:oss iniz}
	{\rm
		(i) It is easy to show that, with the normalization adopted in \eqref{eq:delta}, the mean value of $|\na u|$ over $\pa\Om$ is $1$, that is, $\int_{\pa\Om} | \na u | \, dS_x = | \pa\Om |$.

		(ii)(Asymptotic expansion). As in \cite{EPS}, we deduce from \cite{KV} (see also \cite{SerrinSingularities})
		that the asymptotic behaviour of the solution $u$ of \eqref{eq:delta} near the origin is given by
		\begin{equation}
			\label{eq:asymptoticu}
			u(x) =
			\begin{cases}
				\frac{1}{N-2} \left( \frac{|\pa\Om|}{ \om_N} \right) \, |x|^{- (N-2) } + o \left( |x|^{ -(N-2)} \right) \quad & \text{ if } N>2 ,
				\\
				- \left( \frac{|\pa\Om|}{ \om_N} \right) \, \log |x| + O(1) \quad & \text{ if } N=2  ,
			\end{cases}
		\end{equation}	
	where $\om_N$ denotes the surface measure of the unit sphere in $\RR^N$.	
	}
\end{rem}

\bigskip

If a Serrin-type overdetermined condition requiring $|\na u|$ to be constant on $\partial\Omega$, then, in light of (i), such an overdetermined condition must be
\begin{equation}\label{eq:interno-overdetermination}
	| \na u| \equiv 1  \mbox{ on } \pa\Om .
\end{equation}

In order to state the main result of the present section, we set
\begin{equation}
	\cF(\Om):= \min_{z \in \RR^N} \left\lbrace \frac{| \Om \De B_r(z)|}{r^N} \, : \, |\Om|=|B_r|  \right\rbrace ,
\end{equation}
where $B_r(z)$ denotes the ball of radius $r$ centered at $z$.

\begin{thm}\label{thm:stability delta}
	Let $\Om \subset \RR^N$ be a bounded domain of class $C^{1,\al}$, $0<\al<1$, and let $u$ be solution of \eqref{eq:interno-overdetermination}. 
	
	There exists a constant $C(N)$ such that
	\begin{equation}\label{eq:stability result}
		\cF(\Om) \le C(N) \, \sqrt{ \frac{1}{| \pa\Om |} \int_{\pa\Om} \left( \cM - |\na u| \right) dS_x }  \, ,
	\end{equation}
	where we have set
	\begin{equation*}
		\cM := \max_{\pa \Om} |\na u|.
	\end{equation*}
\end{thm}
\begin{proof}
	(the case $N>2$).	
	As in \cite{EPS}, we consider the P-function 
	$$P= \frac{|\na u|^2}{u^{\frac{2(N-1)}{N-2} }} \, ; $$
	by virtue of the
	harmonicity of $u$ in $\Om \setminus \lbrace 0 \rbrace$,
	\cite[Lemma 1]{Philippin} ensures that the strong maximum principle holds for $P$ in $\Om \setminus \lbrace 0 \rbrace$.
	
	Notice that, if $v$ is solution of $\De v = |\pa \Om| \de_0$ in $\Om$, $v=0$ on $\pa \Om$, then, for any constant $c$, $u = v+c$ satisfies \eqref{eq:delta} and $\na u = \na v$. Hence, we can set
	\begin{equation*}
		c := \frac{ \cM^{\frac{N-2}{N-1}}}{N-2} \left( \frac{| \pa\Om |}{ \om_N } \right)^{\frac{1}{N-1}}
	\end{equation*}
	so that
	\begin{equation*}
		P(0) = (N-2)^{\frac{2(N-1)}{N-2} } \left( \frac{ \om_N }{| \pa\Om |} \right)^{ \frac{2}{N-2}} \ge \frac{|\na u (x)|^2}{c^{\frac{2(N-1)}{N-2} }} = P(x) \quad \text{ for } x \in \pa\Om ,
	\end{equation*}
	where the first identity follows form \eqref{eq:asymptoticu}.
	Hence, by the maximum principle for $P$,
	\begin{equation}\label{eq:sup in P0}
		\sup_\Om P = P(0).
	\end{equation}
	Since a direct computation shows that 
	\begin{equation*}
		- \dv \left( u^{-\frac{N}{N-2} } \na u \right) = \frac{N}{N-2} P \quad \quad \text{ in } \Om\setminus \left\lbrace 0 \right\rbrace ,
	\end{equation*}
	we readily obtain
	$$
	\int_{\Om} P \, dx = \frac{N-2}{N} c^{ - \frac{N}{N-2}} \int_{\pa\Om} (-u_\nu) \, dS_x
	= \frac{N-2}{N} c^{ - \frac{N}{N-2}} | \pa\Om | ,
	$$
	where the last identity follows by (i) of Remark \ref{rem:oss iniz}. Hence, by direct computation and using that $ |B_1| = \om_N /N $, we find that
	\begin{equation*}
		\frac{1}{| \Om |} \int_{\Om} \left[ P(0) - P(x) \right] dx
		=
		\left( \frac{(N-2)^{N-1} \om_N }{ | \pa \Om | } \right)^{\frac{2}{N-2}} \left\lbrace 1 - \left[ \frac{| \pa\Om |}{N | B_1 |^{\frac{1}{N}} |\Om|^{\frac{N-1}{N}}} \right]^{\frac{N}{N-1}} \, \cM^{- \frac{N}{N-1}} \right\rbrace , 
	\end{equation*}
	and after simple manipulations,
	\begin{multline}\label{eq:intermedia}
		\left( \frac{ | \pa \Om | }{(N-2)^{N-1} \om_N } \right)^{\frac{2}{N-2}}	\frac{1}{| \Om |} \int_{\Om} \left[ P(0) - P(x) \right] dx + \cD(\Om)
		\\
		=
		\left[ 1 + \cD(\Om ) \right] 
		\left\lbrace 1 - \frac{ \left[ 1 + \cD(\Om) \right]^{\frac{1}{N-1}}}{\cM^{ \frac{N}{N-1}}} \right\rbrace ,
	\end{multline}
	where we have set
	\begin{equation}\label{eq:isoperimetric deficit}
		\cD(\Om) := \frac{| \pa\Om |}{N | B_1 |^{\frac{1}{N}} |\Om|^{\frac{N-1}{N}}} - 1 ,
	\end{equation}
	that is the (scaling invariant) isoperimetric deficit. Being as $\cD(\Om) \ge 0$ (by the isoperimetric inequality), $\cM \ge 1$ (by (i) of Remark \ref{rem:oss iniz}), and $\frac{N}{N-1}\le 2$, we compute that
	\begin{equation*}
		1 - \frac{ \left[ 1 + \cD(\Om) \right]^{\frac{1}{N-1}}}{\cM^{ \frac{N}{N-1}}} \le \frac{\cM^{2} -1 }{\cM^{2}} \le  2   \left( \cM - 1 \right) ,
	\end{equation*}
	which plugged into \eqref{eq:intermedia} gives
	\begin{equation*}
		\left( \frac{ | \pa \Om | }{(N-2)^{N-1} \om_N } \right)^{\frac{2}{N-2}}	\frac{1}{| \Om |} \int_{\Om} \left[ P(0) - P(x) \right] dx + \left[ 1 - 2   \left( \cM - 1 \right)  \right] \cD(\Om) = 2   \left( \cM - 1 \right).
	\end{equation*} 
	Now we notice that we can assume 
	\begin{equation}\label{eq:small assumption}
		\cM - 1  \le \frac{1}{4},
	\end{equation}
	otherwise \eqref{eq:stability result} trivially holds true being as 
	$$\cF(\Om)\le 2 | B_1 | \le 8 | B_1 | \left( \cM - 1 \right) $$
	and noting that
	\begin{equation}\label{eq:M-1 = integral}
		\cM -1 = \frac{1}{| \pa\Om |} \int_{\pa\Om} \left( \cM - |\na u| \right) dS_x .
	\end{equation}
	Hence, we assume \eqref{eq:small assumption} and find that
	\begin{equation*}
		\left( \frac{ | \pa \Om | }{(N-2)^{N-1} \om_N } \right)^{\frac{2}{N-2}}	\frac{1}{| \Om |} \int_{\Om} \left[ P(0) - P(x) \right] dx + \frac{\cD(\Om)}{2} \le 2   \left( \cM - 1 \right) ,
	\end{equation*} 
	from which, recalling \eqref{eq:sup in P0}, we find that
	\begin{equation*}
		\cD(\Om) \le 4   \left( \cM - 1 \right) , 	
	\end{equation*} 
	and the desired result follows from \cite[Theorem 1.1]{FuscoMP} and \eqref{eq:M-1 = integral}.
	
	\medskip
	
	(The case $N=2$). The case $N=2$ is simpler, as it is entirely based on the isoperimetric inequality and does not require the use of a $P$ function. Being as $N=2$, by the harmonicity of $u$ in $\Om \setminus \left\lbrace 0 \right\rbrace$, we have that the vector field
	\begin{equation*}
		X:= 2 \langle x , \na u \rangle \na u - | \na u |^2 x 
	\end{equation*}
	is divergence free in $\Om \setminus \left\lbrace 0 \right\rbrace$.
	Hence, by the divergence theorem, we compute that
	$$
	\frac{| \pa\Om|^2}{\om_N} = \lim_{\ve \to 0} \int_{\pa B_\ve (0)} \langle X, \nu \rangle \, dS_x =
	\int_{\pa \Om} \langle X, \nu \rangle \, dS_x = \int_{\pa \Om} u_\nu^2 \,  \langle x, \nu \rangle \, dS_x ,
	$$ 
	where in the first identity we used \eqref{eq:asymptoticu} and in the last identity that $\na u = u_\nu \, \nu$ on $\pa\Om$. Subtracting $2 | \Om |$ from both sides and recalling \eqref{eq:isoperimetric deficit}, we immediately obtain
	$$2 | \Om | \, \cD(\Om) = \int_{\pa \Om} \left[ u_\nu^2 - 1 \right] \,  \langle x, \nu \rangle \, dS_x ,$$
	and hence, recalling the definition of $\cM$,
	\begin{equation*}
		\cD(\Om) \le \cM^2 - 1 \le 2 \cM \left( \cM - 1 \right).
	\end{equation*}
	As noticed before, we can assume \eqref{eq:small assumption} (otherwise \eqref{eq:stability result} trivially holds true) to find that
	$$
	\cD(\Om) \le \frac{5}{2} \left( \cM - 1 \right) ,
	$$
	and the desired result follows from \cite[Theorem 1.1]{FuscoMP} and \eqref{eq:M-1 = integral}.
\end{proof}

\begin{rem}\label{rem:generalizations and comparison with planar case}
	{\rm
		(i) Theorem~\ref{thm:stability delta} can be improved using \cite[Theorem 1.1]{FJ} instead of \cite[Theorem 1.1]{FuscoMP}. Doing so shows that Theorem~\ref{thm:stability delta} remains true with $\cF(\Om)$ replaced by the stronger asymmetry
		\begin{equation*}
			\cA(\Om):= \min_{z \in \RR^N} \left\lbrace \frac{| \Om \De B_r(z)|}{r^N} + \left(\frac{1}{r^{N-1}} \int_{\pa\Om} | \nu_\Om(x) - \nu_{B_r(z)}(\pi_{z,r}(x))|^2 \, dS_x   \right) \, : \, |\Om|=|B_r|  \right\rbrace ,
		\end{equation*}
		where $\pi_{z,r}$ is the projection of $\RR^N \setminus \left\lbrace z \right\rbrace$ onto 
		$\pa B_r(z)$, that is,
		\begin{equation*}
			\pi_{z,r}(x) := z + r \frac{x-z}{| x-z |} \quad \text{for all } x \neq z .
		\end{equation*}
		
		(ii) We mention that even though \cite{EPS} treats the problem for the $p$-Laplacian, for the sake of simplicity we decided to treat the case $p=2$ only. Nevertheless, using the same ideas and following \cite{EPS}, Theorem~\ref{thm:stability delta} can easily be extended to the case $2 \neq p \le N$.
		We mention that Theorem~\ref{thm:stability delta} may easily be further extended to the case of the anisotropic $p$-Laplacian, based on the anisotropic extension of \cite{EPS} provided in \cite{XY} and using \cite[Theorem 1.1]{FigalliMP} instead of \cite[Theorem 1.1]{FuscoMP}.
		
		(iii) As already mentioned, in the particular case $N=2$ rigidity and stability were achieved in \cite{AM1, AM2} with a completely different approach, which uses the theory of conformal mappings and is peculiar to the planar case.
		In \cite{AM1}, it is shown that a simply connected domain can be uniquely reconstructed from the knowledge of the normal derivative of its Green's function (with fixed pole). As a consequence, in \cite{AM2}, a stability estimate is obtained around any sufficiently smooth simply connected domain. In particular, in \cite[Theorem 1.1]{AM2}, which may be compared to
		Theorem~\ref{thm:stability delta} (for $N=2$), the closeness to a disk is measured by $\max_{\pa\Om}|x| - \min_{\pa\Om} |x|$, the deviation of $|\na u|$ from a reference constant is measured by a H\"older norm, and the resulting stability estimate is of Lipschitz type.
	}
\end{rem}

\begin{rem}
{\rm 
As already mentioned, Theorem \ref{thm:stability delta} 
makes evident that the overdetermined problem given by \eqref{eq:delta} under the Serrin-type condition $|\na u|=const.$ on $\pa\Om$ is essentially different to the classical Serrin problem, in terms of stability behaviour. In fact, when studying the stability for the classical Serrin problem ``bubbling'' phenomena may arise (see \cite[Theorem 1]{BNST}), and quantitative proximity of $\Om$ to a single ball can be achieved only under some uniform assumption preventing the bubbling (as in \cite[Theorem 1]{ABR}, \cite[Theorem 2]{BNST}, \cite[Theorem 1.1]{CMV}, \cite[Theorem 1]{Feldman} \cite[Theorem 1.1]{MP2}, \cite[Theorem 3.1]{MP3}, \cite[Theorem 4.4]{MP6}). On the contrary, Theorem~\ref{thm:stability delta} shows that, in terms of stability and hence bubbling phenomena, Problem \eqref{eq:delta} behaves like the isoperimetric problem, and, in fact, 
proximity to a single ball is achieved with a constant $C(N)$ that only depends on the dimension.
}
\end{rem}

\section{Weighted Poincar\'e inequalities}\label{sec:Poincare}
\subsection{Notation and known results}

In what follows, for a set $G \subset \RR^N$ and a function $v: G \to \RR$, $v_G$ denotes the {\it mean value of $v$ in $G$}, that is
$$
v_G= \frac{1}{|G|} \, \int_G v \, dx.
$$
Also, denoting with $\de_{\Ga_0} (x)$ the distance of a point $x$ in $G$ to $\Ga_0 \subseteq \pa G$, for a function $v:\Om \to \RR$ we define
$$
\nr \de_{\Ga_0}^\al \, \na v \nr_{L^p (G)} = \left( \sum_{i=1}^N \nr \de_{\Ga_0}^\al \,  v_i \nr_{L^p (G)}^p \right)^\frac{1}{p}
$$
and
$$
\nr \de_{\Ga_0}^\al \, \na^2 v \nr_{L^p (G)} = \left( \sum_{i,j=1}^N \nr \de_{\Ga_0}^\al \, v_{ij} \nr_{L^p (G)}^p \right)^\frac{1}{p},
$$
for $0 \le \al \le 1$ and $p \in [1, \infty)$.

The following weighted Poincar\'{e}-type identity is due to Boas and Straube (\cite{BS}).

\begin{lem}[\cite{BS}]
	\label{lem:BoasStraube}
	Let $G \subset\RR^N$, $N\ge 2$, be a bounded domain with boundary $\pa G$ of class $C^{0,\al}$,
	$0 \le \al \le 1$, and consider $p \in \left[ 1, \infty \right)$. Then, there exists a positive constant, $\ol{\mu}_{p, \al} (G)$ such that
	\begin{equation}
		\label{eq:BoasStraube-poincare}
		\nr v - v_G \nr_{L^p(G)} \le \ol{\mu}_{p, \al} (G)^{-1} \nr \de_{\pa G}^{\al} \, \na v  \nr_{L^p(G)},
	\end{equation}
	for every function $v \in L^p (G) \cap W^{1,p}_{loc} (G)$.
	\par
	In particular, if $G$ has a Lipschitz boundary, the number $\al$ can be replaced by any exponent in $[0,1]$. 
\end{lem}

\begin{rem}
	{\rm
		When $\al=0$ we understand the boundary of $G$ to be locally the graph of a continuous function. 
	}
\end{rem}

We now present a simple lemma which will be useful in the sequel to manipulate the left-hand side of Poincar\'e-type inequalities.

\begin{lem}\label{lem:mediaor-mediauguale}
	Let $G$ be a domain with finite measure and let $F \subseteq G$ be a set of positive measure. If $v \in L^p(G)$, then for every $\la \in \RR$
	\begin{equation}\label{eq:mediaor-mediauguale}
		\nr v - v_F \nr_{L^p(G) } \le \left[ 1+ \left( \frac{|G|}{|F|} \right)^{\frac{1}{p}} \right] \, \nr v - \la \nr_{L^p (G)}.
	\end{equation}
\end{lem}
\begin{proof}
	By H\"older's inequality, we have that 
	\begin{equation*}
		|v_F - \la| \le \frac1{|F|} \int_F |v-\la| \, dx \le 
		|F|^{-1/p} \nr v - \la \nr_{L^p(F)}. 
	\end{equation*}
	Since $|v_F - \la|$ is constant, we then infer that
	$$
	\nr v_F-\la\nr_{L^p(G)}=|G|^{1/p} |v_F - \la| \le \left( \frac{|G|}{|F|} \right)^{\frac{1}{p}} \, \nr v- \la \nr_{L^p (G)}.
	$$
	Thus, \eqref{eq:mediaor-mediauguale} follows by an application of the triangular inequality.
\end{proof}

When $p(1 -\al)<N$, a strengthened Sobolev-Poincar\'e version of inequality \eqref{eq:BoasStraube-poincare} can be obtained, as shown in the next lemma, which is a reformulation of a result obtained by Hurri-Syrj\"anen in \cite{HS}, which, in turn, was stimulated by \cite{BS}.

\begin{lem}\label{lem:Hurri}
	Let $G \subset \RR^N$ be a bounded $b_0$-John domain, and consider three numbers $r, p, \al$ such that 
	\begin{equation}\label{eq:r p al in Hurri}
		1 \le p \le r \le \frac{Np}{N-p(1 - \al )} , \quad p(1 - \al)<N , 
		\quad 0 \le \al \le 1 .
	\end{equation}
	Then, there exists a positive constant $\ol{\mu}_{r, p, \al} (G)$ such that
	\begin{equation}
		\label{eq:John-Hurri-poincare}
		\nr v - v_G \nr_{L^r (G)} \le \ol{\mu}_{r, p, \al} (G)^{-1} \nr \de_{\pa G}^{\al} \, \na v  \nr_{L^p(G)},
	\end{equation}
	for every function $v\in L^1_{loc}(G)$ such that $\de_{\pa G}^{\al} \, \na v \in L^p (G)$ .
\end{lem}

The class of John domains is huge: it contains Lipschitz domains, but also very irregular domains with fractal boundaries as, e.g., the Koch snowflake.
Roughly speaking, a domain is a $b_0$-John domain if it is possible to travel from one point of the domain to another without going too close to the boundary.
The formal definition is the following: a domain $G$ in $\RR^N$ is a {\it $b_0$-John domain}, $b_0 \ge 1$, if each pair of distinct points $a$ and $b$ in $G$ can be joined by a curve $\ga: \left[0,1 \right] \rightarrow G$ such that
%
%
\begin{equation*}
	\de_{ \pa G} (\ga(t)) \ge b_0^{-1} \min{ \left\lbrace |\ga(t) - a|, |\ga(t) - b| \right\rbrace  }.
\end{equation*}
The notion could be also defined through the so-called {\it $b_0$-cigar} property (see \cite{Va}).

We mention that the proof in \cite{HS}
%
%
allows to obtain explicit upper bounds for $\ol{\mu}_{r, p, \al} (G)^{-1}$ in terms of ($N$, $r$, $p$, $\al$,) $b_0$ and $|G|$ only.
\begin{proof}[Proof of Lemma \ref{lem:Hurri}]
	In \cite[Theorem 1.3]{HS} it is proved that there exists a constant $c=c(N,\, r, \, p,\, \al, \, G)$ such that
	\begin{equation}\label{eq:risultatodiHSconr-media}
		\nr v - v_{r, G} \nr_{L^r (G)} \le c \, \nr \de_{\pa G}^{\al} \, \na v  \nr_{L^p (G)} ,
	\end{equation}
	for every $v \in L^1_{loc}(G)$ such that $\de_{ \pa G}^\al \, \na v  \in L^p(G)$.
	Here, $v_{r,G}$ denotes the {\it $r$-mean of $v$ in $G$} which is defined -- following \cite{IMW} -- as the unique minimizer of the problem
	$$\inf_{\la \in \RR} \nr v - \la \nr_{L^r(G)}.$$
	Notice that, in the case $r=2$, $v_{2, G }$ is the classical mean value of $v$ in $G$, i.e. $v_{2, G } = v_G$,
	as can be easily verified.
	
	Using Lemma \ref{lem:mediaor-mediauguale} with $F= G$ and $\la = v_{r, G}$, \eqref{eq:John-Hurri-poincare} follows.
\end{proof}

\subsection{Weighted Poincar\'e-type inequalities for vector fields}\label{subsec:Poincare-type}
In the present subsection, we provide new weighted Poincar\'e-type inequalities for vector fields. Their strengthened Sobolev-Poincar\'e versions will be presented later in Subsection~\ref{subsec:SobolevPoincare versions}. 
\begin{thm}\label{thm:Poincare new in general}
	Given $1 \le p < +\infty$ and $0 \le \al < (p-1)/p$, let $G \subset \RR^N$ be a bounded Lipschitz
	domain. 
	Let $A \subseteq \pa G$ be a relatively open subset of $\pa G$ with positive $N-1$-dimensional measure and such that the (exterior) unit normal vector $\nu(x)$ is continuous on $A$. 
	
	Then, there exists a positive constant $ \eta_{p, \al}(A , G)$
	such that
	\begin{equation}\label{eq:Poincare new RN}
		\nr \vV \nr_{L^p(G)} \le \eta_{p, \al}(A , G)^{-1} \, \nr \de_{ \pa G}^\al D \vV \nr_{L^p(G)} ,
	\end{equation}
	for any vector function
	$$\vV: G \to \mathrm{span}_{x \in A} \, \nu (x) \subseteq \RR^N$$
	belonging to $W^{1,p}_{\al} (G)$
	such that $\langle \vV , \nu \rangle = 0$ a.e. on $A$. Here, $W^{1,p}_{\al} (G)$ denotes the weighted Sobolev space with norm given by $ \nr \vV \nr_{L^p(G)} + \nr \de_{\pa G}^\al D \vV \nr_{L^p(G)} $.
\end{thm}
	\begin{proof}
		By following \cite{BS}, the proof results combining Hardy's inequality with a standard compactness argument that goes back at least to Morrey \cite{Morrey}. As in \cite{BS}, the Hardy inequality that we need is the following, which can be deduced by \cite{Kufner}: there exists a smoothly bounded relatively compact subdomain $\omega$ of $G$ such that
		\begin{equation}\label{eq:1_PoincareBS}
			\nr \vV \nr_{L^p( G )} \le C \left( \nr \de_{\pa G}^\al D \vV \nr_{L^p(G)} + \nr \vV \nr_{L^p( \omega )} \right),
		\end{equation}
		for every locally integrable $\vV$.
		Hence, to prove the theorem we only need to show that
		\begin{equation}\label{eq:ineq intermedia per Poinc pesata}
		\nr \vV \nr_{L^p( \omega )} \le C \nr \de_{ \pa G}^\al D \vV \nr_{L^p( G )}.
		\end{equation}
		Suppose this estimate to be false. Then there would exist a sequence $\left\lbrace \vV_k \right\rbrace$ contained in 
		$$
		\left\lbrace \vV: G \to \mathrm{span}_{x \in A} \, \nu (x) \subseteq \RR^N, \, \vV\in W^{1,p}_\al (G) , \, \langle \vV , \nu \rangle = 0 \, \text{ a.e. on } A \right\rbrace
		$$
		such that
		\begin{equation}\label{eq:2_PoincareBS}
			\nr \vV_k \nr_{ L^p (\omega)} = 1
		\end{equation}
		and
		\begin{equation}\label{eq:3_PoincareBS}
			\nr \de_{ \pa G }^\al D \vV_k \nr_{L^p(G)} < 1/k \quad \text{for each } k .
		\end{equation}
		In particular, the $\vV_k$ form a bounded sequence in $W^{1,p}(\om)$, which embeds compactly in $L^p(\omega)$ by the Rellich-Kondrachov theorem (see, e.g., \cite[Theorem 6.2]{Adams}).
		By passing to a subsequence we may assume that the $\vV_k$ converge in $L^p (\omega)$ to some limit $\vV_0$, and
		in view of \eqref{eq:1_PoincareBS} and \eqref{eq:3_PoincareBS} the convergence even takes place in $W^{1,p}_{\al} (G)$.
		But $\nr \vV_0 \nr_{W^{1,p} ( \om )} = 1$ by \eqref{eq:2_PoincareBS}, and $D \vV_0$ vanishes identically (a.e. in $G$) by \eqref{eq:3_PoincareBS}. Hence, $\vV_0$ is a nonzero constant vector in $G$.

		The above argument works for any $0 \le \al \le 1$ provided that the boundary $\pa G$ of $G$ is of class $C^{0,\al}$. 
		Let us now exploit the stronger assumption on $\al$ in Theorem~\ref{thm:Poincare new in general}.
		
		In fact, being as $0 \le \al < (p-1)/p$, by the trace theorem for weighted Sobolev spaces (see, e.g., \cite[Theorem 9.15]{Kufner}) we have that
		\begin{equation}\label{eq:4_PoincareBS}
		\vV_k \to \vV_0 \text{ in } L^p(\pa G) \supseteq L^p(A), \quad \text{ where } \vV_0 \text{ is a nonzero constant vector}.
		\end{equation}
		
		We can now use the assumption $\langle \vV_k, \nu \rangle = 0$ a.e. on $A$, to find a contradiction.
		In fact, by using such assumption 
		%
		%
		we have that
		\begin{equation}\label{eq:serve solo per alternativa bis}
	    | \langle \vV_0, \nu(x) \rangle | = |\langle \vV_k(x) - \vV_0, \nu(x) \rangle| \le  |\vV_k(x) - \vV_0|  
	    \end{equation}
	    for a.e. $x$ in $A$,
		and hence $\langle \vV_0, \nu \rangle = 0$ a.e. on $A$.

		By continuity of $\nu$ on $A$, we actually have that
		\begin{equation}\label{eq:provamethodPoincare}
			\langle \vV_0, \nu \rangle = 0 \quad \text{everywhere in } A .
		\end{equation}
		Since $\vV_0 \in \mathrm{span}_{x \in A} \nu (x)$,
		\eqref{eq:provamethodPoincare} gives that $\vV_0$ must be the zero vector. This contradicts \eqref{eq:4_PoincareBS} and proves the theorem.
	\end{proof}

\begin{rem}
	{\rm
		
			(i) In the particular case where
			\begin{equation}\label{eq:condition Poincare span RN}
				\mathrm{span}_{x \in A} \, \nu (x) = \RR^N ,
			\end{equation}
			inequality \eqref{eq:Poincare new RN} holds true for any vector function $\vV: G \to \RR^N$ belonging to $W^{1,p}_{\al} (G)$ such that $\langle \vV , \nu \rangle = 0$ a.e. on $A$.

			We notice that a point of strict convexity in $A$ is sufficient to guarantee the validity of \eqref{eq:condition Poincare span RN}.

(ii) When the dimension of $\mathrm{span}_{x \in A} \, \nu (x)$ equals $1$, Theorem \ref{thm:Poincare new in general} reduces to a weighted Poincar\'e-type inequality for scalar valued functions with zero trace on a subset of the boundary with positive $N-1$-dimensional measure.
In this case, the constant $\eta_{p, \al}(A , G)$ can be explicitly estimated in terms of $N$, $p$, $\al$, $|G|/|A|$ (where $|G|$ and $|A|$ denote the $N$-dimensional measure of $G$ and the $N-1$-dimensional measure of $A$), the constant $\ol{\mu}_{p, \al}(G)^{-1}$ in the Poincar\'e inequality \eqref{eq:BoasStraube-poincare} and the constant $\la_{p,\al}(A)$ of the trace inequality of the embedding $W^{1,p}_{\al}(G) \hookrightarrow L^p(A)$.  Namely, by following the argument in \cite[Lemma 3.10]{Pog3} and adapting it to the weighted setting, we obtain that
$$
\eta_{p, \al}(A , G)^{-1} \le \ol{\mu}_{p,\al}(G)^{-1} + \left( \frac{|G|}{|A|} \right)^{\frac{1}{p}} \, \la_{p , \al }(A) \left( 1 + \ol{\mu}_{p,\al}(G)^{-p} \right)^{1/p} .
$$
Explicit estimates in this case may also be obtained by using \cite{Me} (see also \cite{AMR}).
}
\end{rem}

\medskip

Theorem \ref{thm:Poincare new in general} provides a new crucial tool, which is used (in the unweighted setting $\al=0$) by the author in \cite{Pog3}, to obtain quantitative rigidity results for a class of problems relative to convex cones, including Alexandrov's soap bubble-type theorems and Heintze-Karcher-type inequalities relative to convex cones. As mentioned in \cite{Pog3}, a weighted extension of such a Poincar\'e-type inequality is needed to obtain analogous quantitative rigidity results for mixed Serrin-type overdetermined problems relative to convex cones, which will be addressed in \cite{PPR}. In view of this application, the following variant of Theorem \ref{thm:Poincare new in general} will also be useful. Such a variant, which will be used in \cite{PPR}, allows $\al$ greater than the critical value $(p-1)/p$ at the cost of replacing $\de_{\pa G}$ with $\de_{\Ga_0}$, for $\Ga_0 \subseteq \pa G \setminus \ol{A}$.

\begin{thm}\label{thm:Poincare new RN_WEIGHTS FOR SERRIN}
	Given $1 \le p < +\infty$ and $0 \le \al \le 1$, let $G \subset \RR^N$ be a bounded	domain with boundary $\pa G$ of class $C^{0,\al}$. 
	Let $A \subseteq \pa G$ be a relatively open subset of $\pa G$ with positive $N-1$-dimensional measure and such that $\nu(x)$ is continuous on $A$.
	Let $\Ga_0 \subseteq \pa G \setminus \ol{A} $.
	
	Then, there exists a positive constant $ \eta_{p, \al}(A , \Ga_0, G)$
	such that
	\begin{equation}\label{eq:Poincare new RN_WEIGHTS FOR SERRIN}
		\nr \vV \nr_{L^p(G)} \le \eta_{p, \al}(A , \Ga_0, G)^{-1} \, \nr \de_{ \Ga_0 }^\al D \vV \nr_{L^p(G)} ,
	\end{equation}
	for any vector function $\vV: G \to \mathrm{span}_{x \in A} \, \nu (x) \subseteq \RR^N$ belonging to $W^{1,p}_{\Ga_0 , \al} (G)$
	such that $\langle \vV , \nu \rangle = 0$ a.e. on $A$.
	Here, $W^{1,p}_{\Ga_0 , \al} (G)$ denotes the weighted Sobolev space with norm given by $ \nr \vV \nr_{L^p(G)} + \nr \de_{\Ga_0}^\al D \vV \nr_{L^p(G)} $.
\end{thm}
\begin{proof}
The same arguments used at the beginning of the proof of Theorem \ref{thm:Poincare new in general} give that there exists a smoothly bounded relatively compact subdomain $\omega$ of $G$ such that
\begin{equation*}
	\nr \vV \nr_{L^p( G )} \le C \left( \nr \de_{\Ga_0}^\al D \vV \nr_{L^p(G)} + \nr \vV \nr_{L^p( \omega )} \right),
\end{equation*}
for every locally integrable $\vV$.
Hence, to prove the theorem we only need to show that
\begin{equation}\label{eq:ineq intermedia per Poinc pesata_PERSERRIN VERSION}
	\nr \vV \nr_{L^p( \omega )} \le C \nr \de_{ \Ga_0}^\al D \vV \nr_{L^p( G )}.
\end{equation} 
By reasoning as in the proof of Theorem \ref{thm:Poincare new in general}, if we suppose \eqref{eq:ineq intermedia per Poinc pesata_PERSERRIN VERSION} to be false, then we can find a sequence $\vV_k$ contained in
$$
\left\lbrace \vV: G \to \mathrm{span}_{x \in A} \, \nu (x) \subseteq \RR^N, \, \vV\in W^{1,p}_{\Ga_0 , \al} (G) , \, \langle \vV , \nu \rangle = 0 \, \text{ a.e. on } A \right\rbrace ,
$$
satisfying 
%
%
%
$$
\nr \vV_k \nr_{ L^p (\omega)} = 1
\quad \text{and} \quad
\nr \de_{ \Ga_0}^\al D \vV_k \nr_{L^p(G)} < 1/k \quad \text{for each } k , 
$$
and converging in $W^{1,p}_{\Ga_0,\al} (G)$ to some limit $\vV_0$, which is a nonzero constant vector in $G$.

Since $\vV_k \in W^{1,p}_{\Ga_0, \al} (G)$ and $\langle \vV_k ,\nu \rangle = 0 $ a.e. on $A$, for any $\varphi \in C^1_c ( \RR^N \setminus \ol{A} )$ 
we have that
\begin{equation*}
	\begin{split}
	\int_{A} \langle \vV_0 , \nu \rangle \, \varphi  \, dSx
	& = \int_{G} \langle \vV_0 , \na \varphi \rangle \, dx + \int_{G} \dv( \vV_0) \, \varphi \, dx 
	\\ 
	& = \lim_{k\to \infty} \left\lbrace \int_{G} \langle \vV_k , \na \varphi \rangle \, dx + \int_{G} \dv( \vV_k) \, \varphi \, dx \right\rbrace = 
	0,
	\end{split}
\end{equation*}
and hence, 
\begin{equation}\label{eq:prova solo per alternativa}
\langle \vV_0, \nu \rangle = 0 \text{ a.e. in } A.
\end{equation}	

We can now conclude as before.
In fact, by continuity of $\nu$ on $A$, we actually have \eqref{eq:provamethodPoincare}, and being as $\vV_0 \in \mathrm{span}_{x \in A} \nu (x)$, \eqref{eq:provamethodPoincare} gives that $\vV_0$ must be the zero vector.
This gives a contradiction and proves the theorem.
\end{proof}

\begin{rem}\label{rem:alternative}
{\rm 
An alternative way to achieve \eqref{eq:prova solo per alternativa} in the proof above is to notice that 
\begin{equation}\label{eq:alternativa}
\vV_k \to \vV_0 \text{ in } L^p(B) \text{ for any open set } B\subset \subset A ,
\end{equation}
by the classical (unweighted) Sobolev trace embedding, and then reason as in \eqref{eq:serve solo per alternativa bis}. In fact,
%
for any smooth open set $\om \subset G$ such that $B\subset \ol{\om}$ and $\mathrm{dist}(\Ga_0, \om)>0$,
the classical Sobolev trace embedding gives that $\nr \vV_k - \vV_0 \nr_{L^p(B)} \le C \nr \vV_k - \vV_0 \nr_{W^{1,p}(\om)}$, so that we easily get that
$\nr \vV_k - \vV_0 \nr_{L^p(B)} \le C \nr \vV_k - \vV_0 \nr_{W^{1,p}_{\Ga_0,\al}(G)}$ and hence \eqref{eq:alternativa}.
}
\end{rem}

\subsection{Strengthened Sobolev-Poincar\'e versions}\label{subsec:SobolevPoincare versions}

As a consequence of Lemma \ref{lem:Hurri}, we have the following weighted version of Sobolev inequality.

\begin{cor}\label{cor:Sobolev weighted}
	Let $G \subset \RR^N$ be a bounded $b_0$-John domain, and consider three numbers $r, p, \al$ satisfying \eqref{eq:r p al in Hurri}.
	Then, for any function $v\in L^1_{loc}(G)$ such that $\de_{\pa G}^{\al} \, \na v \in L^p (G)$ we have that
	\begin{equation}\label{eq:Sobolev weighted}
		\nr v  \nr_{L^r (G)} \le 
		|G|^{ \frac{1}{r} - \frac{1}{p} } \, \nr v \nr_{L^p(G)} +  \ol{\mu}_{r, p, \al} (G)^{-1} \, \nr \de_{\pa G}^{\al} \, \na v  \nr_{L^p(G)}
	\end{equation}
	where $\ol{\mu}_{r, p, \al} (G)^{-1}$ is that appearing in Lemma \ref{lem:Hurri}.
\end{cor}
\begin{proof}
	For any function $v\in L^1_{loc}(G)$ such that $\de_{\pa G}^{\al} \, \na v \in L^p (G)$, we compute that
	\begin{equation*}
		\begin{split}
			\nr v \nr_{L^r(G)} 
			& \le 
			\nr v_G \nr_{L^r(G)} + \nr v - v_G \nr_{L^r(G)}
			\\
			& = 
			|G|^{ \frac{1}{r} - 1 } \,  \left| \int_G v \, dx \right| + \nr v - v_G \nr_{L^r(G)}
			\\
			& \le 
			|G|^{ \frac{1}{r} - \frac{1}{p} } \, \nr v \nr_{L^p(G)} +  \ol{\mu}_{r, p, \al} (G)^{-1} \, \nr \de_{\pa G}^{\al} \, \na v  \nr_{L^p(G)} ,
		\end{split}
	\end{equation*}
	where in the last inequality we used H\"older's inequality for the first summand and Lemma \ref{lem:Hurri} for the second summand.
\end{proof}

As a consequence, in the case $p( 1 - \al) < N$, we have the following strengthened Sobolev-Poincar\'e versions of \eqref{eq:Poincare new RN} and \eqref{eq:Poincare new RN_WEIGHTS FOR SERRIN}. 
\begin{thm}\label{thm:Strengthened Poincare new RN}
	let $G \subset \RR^N$ be a bounded Lipschitz domain. 
	Let $A \subseteq \pa G$ be a relatively open subset of $\pa G$ with positive $N-1$-dimensional measure and such that $\nu(x)$ is continuous on $A$. 
    Let $r,p,\al$ three numbers such that
	$$
	1 \le p \le r \le \frac{Np}{N-p(1 - \al )} , \quad p(1 - \al)<N .
	\quad 0 \le \al < \frac{(p-1)}{p}
	$$

	Then, there exists a positive constant $ \eta_{r,p,\al}(A, G)$ 
	such that
	\begin{equation}\label{eq:Strengthened Poincare new RN}
		\nr \vV \nr_{L^{r}(G)} \le \eta_{r,p,\al}(A, G)^{-1} \, \nr \de_{ \pa G }^\al D \vV \nr_{L^p(G)} ,
	\end{equation}
	for any vector function
	$\vV: G \to \mathrm{span}_{x \in A} \, \nu (x) \subseteq \RR^N$ belonging to $W^{1,p}_{\al} (G)$ such that $\langle \vV , \nu \rangle = 0$ a.e. on $A$.
\end{thm}


\begin{thm}\label{thm:StrengthenedSobolevPoincareWEIGHTSFORSERRIN}
Let $G \subset \RR^N$ be a bounded	domain with boundary $\pa G$ of class $C^{0,\al}$. 
Let $A \subseteq \pa G$ be a relatively open subset of $\pa G$ with positive $N-1$-dimensional measure and such that $\nu(x)$ is continuous on $A$.
Let $\Ga_0 \subseteq \pa G \setminus \ol{A} $.
Let $r, p, \al$ be three numbers satisfying \eqref{eq:r p al in Hurri}.

Then, there exists a positive constant $ \eta_{r,p,\al}(A, \Ga_0, G)$ 
such that
\begin{equation}\label{eq:Strengthened Poincare new RN WEIGHTS AS IN SERRIN}
	\nr \vV \nr_{L^{r}(G)} \le \eta_{r,p,\al}(A, \Ga_0, G)^{-1} \, \nr \de_{ \Ga_0}^\al D \vV \nr_{L^p(G)} ,
\end{equation}
for any vector function
$\vV: G \to \mathrm{span}_{x \in A} \, \nu (x) \subseteq \RR^N$ belonging to $W^{1,p}_{\Ga_0,\al} (G)$ such that $\langle \vV , \nu \rangle = 0$ a.e. on $A$.
\end{thm}	
\begin{proof}[Proof of Theorems \ref{thm:Strengthened Poincare new RN} and \ref{thm:StrengthenedSobolevPoincareWEIGHTSFORSERRIN}]
	%
	%
	By using \eqref{eq:Sobolev weighted} on each component of $\vV=(v_1, \dots, v_N)$ we have that
	$$
	\nr v_i  \nr_{L^r ( G )} 
	\le 
	2
	\max\left\lbrace
	| G |^{ \frac{1}{r} - \frac{1}{p} } \, \nr v_i \nr_{L^p( G )} 
	, \,
	\ol{\mu}_{r, p, \al} ( G )^{-1} \, \left( \sum_{j=1}^N \nr \de_{\Ga_0}^{\al} \, (v_i)_j \nr^p_{L^p( G)} \right)^{\frac{1}{p}}
	\right\rbrace ,
	$$
	where we also used that  $ \de_{\pa G }(x) \le \de_{\Ga_0}(x)$.
	Raising to the power of $r$, summing up for $i=1, \dots ,N$, and then raising to the power of $1/r$,
	we obtain that
	$$
	\nr \vV  \nr_{L^r ( G )} 
	\le 
	2
	\max\left\lbrace
	| G |^{ \frac{1}{r} - \frac{1}{p} } \, \nr \vV \nr_{L^p( G )} 
	, \,
	\ol{\mu}_{r, p, \al} ( G )^{-1} \, \nr \de_{\Ga_0}^{\al} \, D \vV \nr^p_{L^p( G )} 
	\right\rbrace ,
	$$
	where we used the inequality
	\begin{equation}\label{eq:ineq for Sobolev norm equivalence}
		\sum_{i=1}^N x_i^{\frac{r}{p}} \le \left( \sum_{i=1}^N x_i \right)^{\frac{r}{p}} 
	\end{equation}
	which holds for every $(x_1, \dots, x_N) \in \RR^N$ with $x_i \ge 0$ for $i=1, \dots, N$, since $r/p \ge 1$.
	
	Thus, \eqref{eq:Strengthened Poincare new RN} in Theorem \ref{thm:Strengthened Poincare new RN} (resp. \eqref{eq:Strengthened Poincare new RN WEIGHTS AS IN SERRIN} in Theorem \ref{thm:StrengthenedSobolevPoincareWEIGHTSFORSERRIN}) follows by \eqref{eq:Poincare new RN} in Theorem \ref{thm:Poincare new in general} (resp. \eqref{eq:Poincare new RN_WEIGHTS FOR SERRIN} in Theorem \ref{thm:Poincare new RN_WEIGHTS FOR SERRIN}).
\end{proof}

\begin{rem}
	{\rm The proof above provides the following explicit upper bounds for the constants in \eqref{eq:Strengthened Poincare new RN} and \eqref{eq:Strengthened Poincare new RN WEIGHTS AS IN SERRIN}.
		
		For the constant $\eta_{r,p,\al}(A , G )^{-1}$ in \eqref{eq:Strengthened Poincare new RN}, we have that
		\begin{equation*}
			\eta_{r,p,\al}(A , G )^{-1} 
			\le 
			\max\left\lbrace 
			| G |^{ \frac{1}{r} - \frac{1}{p} } \, \, \eta_{p,\al}(A, G )^{-1} 
			, \,
			\ol{\mu}_{r, p, \al} ( G )^{-1} 
			\right\rbrace ,
		\end{equation*}
		where $\ol{\mu}_{r, p, \al} (G)^{-1}$ and $\eta_{p,\al}(A, G)^{-1}$ are those appearing in Lemma \ref{lem:Hurri} and Theorem~\ref{thm:Poincare new in general}.
		
		Analogously, for the constant $\eta_{r,p,\al}(A , \Ga_0, G )^{-1}$ in \eqref{eq:Strengthened Poincare new RN WEIGHTS AS IN SERRIN} we have that 
		\begin{equation*}
			\eta_{r,p,\al}(A , \Ga_0, G )^{-1} 
			\le 
			\max\left\lbrace 
			| G |^{ \frac{1}{r} - \frac{1}{p} } \, \, \eta_{p,\al}(A, \Ga_0, G )^{-1} 
			, \,
			\ol{\mu}_{r, p, \al} ( G )^{-1} 
			\right\rbrace ,
		\end{equation*}
		where $\ol{\mu}_{r, p, \al} (G)^{-1}$ and $\eta_{p,\al}(A, \Ga_0, G)^{-1}$ are those appearing in Lemma \ref{lem:Hurri} and Theorem \ref{thm:Poincare new RN_WEIGHTS FOR SERRIN}.	
	}
\end{rem}


\section{Harmonic functions in cones: mean value property, a weighted Poincar\'e inequality, and a duality theorem}\label{sec:cones}

In Subsection \ref{subsec:harmonic functions and Poincare inequality} we provide a mean value-type property and an associated weighted Poincar\'e-type inequality for harmonic functions in cones.
In Subsection~\ref{subsec:duality} we provide a duality relation between this new mean value property and a partially overdetermined problem in cones.

To this aim, we now introduce the following Setting.

\begin{setting}\label{Setting}
	{\rm
		In what follows, $\Si$ denotes a convex cone in $\RR^N$ with vertex at the origin, that is
		$\Si=\left\lbrace tx \, : \, x \in \om, \, t\in ( 0 , +\infty)  \right\rbrace$, for some open domain $\om \subseteq \mathbb{S}^{N-1}$. We consider a bounded domain $\Si\cap\Om$ – where $\Om$ is a smooth bounded domain in $\RR^N$ – such that its boundary relative to the cone $\Ga_0 := \Si \cap \pa\Om$ is smooth, and we set $\Ga_1:= \pa(\Si \cap \Om)\setminus \ol{\Ga}_0$.
	}
\end{setting}

	\subsection{Harmonic functions in cones}\label{subsec:harmonic functions and Poincare inequality}
	
	The next result provides a mean value-type property for harmonic functions in cones. In particular, we consider functions satisfying
		\begin{equation}\label{eq:problem harmonic}
		\begin{cases}
			\De v = 0 \quad & \text{ in } \Si\cap\Om
			\\
			v_\nu=0 \quad & \text{ on } \Ga_1 .
		\end{cases}
	\end{equation}
	We recall that if $v$ is a weak solution to \eqref{eq:problem harmonic}, then it belongs to $W^{2,2}_{loc} \left( \left(\Si \cap \Om \right) \cup \Ga_1 \right)$ (see, e.g., \cite{AJ, CFR, CL}).
	\begin{thm}[Mean value property for harmonic functions in cones]
		\label{thm:mean value property harmonic in cones}
		Let $\Si\cap\Om$ be as in Setting \ref{Setting} and assume that $\pa(\Si\cap\Om) \setminus \ol{\Ga}_0$ contains a vertex $x_0$ of the cone $\Si$.
		Let $v \in C^0 (\ol{\Si}\cap\Om)$ 
		satisfy \eqref{eq:problem harmonic}.
		Then, the following mean value properties hold:
		\begin{equation}\label{eq:boundary mean value harmonic}
			v(x_0) = \frac{1}{| \Si \cap \pa B_r (x_0) |} \int_{\Si \cap \pa B_r (x_0) } v(y) \, dS_y ,
			\quad
			\text{for any } \, 0 \le r <\de_{\Ga_0}(x_0)
		\end{equation}
		\begin{equation}\label{eq:mean value harmonic}
			v(x_0) = \frac{1}{| \Si \cap B_r (x_0) |} \int_{\Si \cap B_r (x_0) } v(y) \, dy ,
			\quad \quad
			\text{for any } \, 0 \le r \le \de_{\Ga_0}(x_0)
		\end{equation}
	\end{thm}
	\begin{proof}
		We first show that \eqref{eq:boundary mean value harmonic} and \eqref{eq:mean value harmonic} are equivalent. In fact, if \eqref{eq:boundary mean value harmonic} holds true, then
		\begin{equation*}
			\begin{split}
			\int_{ \Si \cap B_r (x_0)} v(y) \, dy 
			& = \int_0^r \left( \int_{  \Si \cap \pa B_s (x_0) } v(y) \, dS_y \right) ds 
			\\
			& = v(x_0) \int_0^r | \Si  \cap \pa B_s (x_0)| \, ds 
			\\
			& = v(x_0) | \Si \cap B_r (x_0)  |.
			\end{split}
		\end{equation*}
		The converse can be proved by differentiating with respect to $r$ the identity:
		$$
		\int_{ \Si \cap B_r (x_0) } v(y) \, dy = | \Si \cap B_r (x_0)| \, v(x_0) .
		$$
		
		Hence, it is enough to prove that if $v$ satisfies \eqref{eq:problem harmonic}, then \eqref{eq:boundary mean value harmonic} holds true.
		Since $x_0$ is a vertex of $\Si$, without loss of generality we can set $x_0$ to be the origin, and hence denote simply with $B_r$ and $\pa B_r$ the ball and the sphere of radius $r$ centered at $x_0=0$.
		We set
		$$
		\psi (r) = \frac{1}{|  \Si \cap \pa B_r   |} \int_{ \Si \cap \pa B_r  } v(y) \, dS_y = \frac{1}{|  \Si \cap \pa B_1 |} \int_{ \Si \cap \pa B_1  } v(r \, y) \, dS_y,
		$$
		and compute
		\begin{equation*}
			\begin{split}
			\psi' (r) 
			& = \frac{1}{| \Si \cap \pa B_1  |} \int_{ \Si \cap \pa B_1  } \langle \na v (r \, y), y \rangle \, dS_y 
			\\
			& = \frac{1}{| \Si \cap \pa B_r |} 
			\int_{  \Si \cap \pa B_r } v_\nu(y) \, dS_y
			\\
			& = 0,
			\end{split}
		\end{equation*}
		where the last equality follows from \eqref{eq:problem harmonic} and the divergence theorem.
		Thus, $\psi$ is constant and we have that $\psi(r)= \psi(0^+)=v(x_0)$.
	\end{proof}
	
	We now show how the mean value property presented above can be used to obtain a constructive proof of the following Poincar\'e-type inequality for harmonic functions in cones.
	To this aim, we are going to use Lemma \ref{lem:mediaor-mediauguale}.

	\begin{lem}\label{lem:Poincare harmonic in cone}
		Let $\Si\cap\Om \subset \RR^N$ be as in Setting \ref{Setting} and
		assume that $\pa(\Si\cap\Om) \setminus \ol{\Ga}_0$ contains a vertex $x_0$ of the cone $\Si$.
		Let $r, p, \al$ be three numbers and assume that either 
		\begin{equation}\label{eq:case r=p}	
		1 \le r=p < \infty , \quad 0 \le \al \le 1, \quad \text{and the boundary of } \Si\cap\Om \text{ is of class } C^{0,\al}
		\end{equation}
		or 
		\begin{equation}\label{eq:prova caso come per John incluso John}
		r, p, \al \text{ are as in \eqref{eq:r p al in Hurri} and } \Si \cap \Om \text{ is a } b_0\text{-John domain}.
		\end{equation}
		Then, there exists a positive constant $c$ such that
		\begin{equation}
			\label{eq:Poincare per armoniche}
			\nr v - v(x_0) \nr_{L^r ( \Si\cap \Om)} \le c \, \nr  \na v  \nr_{L^p( \Si\cap\Om )},
		\end{equation}
		for every function $v\in C^0 (\ol{\Si}\cap\Om)$ satisfying \eqref{eq:problem harmonic}.
		The constant $c$ can be explicitly estimated by means of
		\begin{equation*}
			c \le \left( 1 + \frac{| \Si\cap\Om|}{| \Si\cap B_{\de_{\Ga_0}(x_0)} (x_0) |} \right)^{\frac{1}{r}} \,
			\begin{cases} 
				\ol{\mu}_{p,\al}(\Si\cap\Om)^{-1} \quad &\text{ in case \eqref{eq:case r=p}} 
				\\
				\ol{\mu}_{r,p, \al}(\Si\cap\Om)^{-1} \quad &\text{ in case } \eqref{eq:prova caso come per John incluso John} .
			\end{cases}	
		\end{equation*}
	\end{lem}
	\begin{proof}
		Since $v$ satisfies \eqref{eq:problem harmonic},
		the mean value property \eqref{eq:mean value harmonic} gives that
		$v(x_0)= v_{\Si\cap B_{\de_{\Ga_0}(x_0)} (x_0)}$, and hence
		\begin{equation*}
			\begin{split}
			\nr v - v(x_0) \nr_{L^r ( \Si\cap \Om)} 
			& = 
			\nr v - v_{\Si\cap B_{\de_{\Ga_0}(x_0)} (x_0)} \nr_{L^r ( \Si\cap \Om)} 
			\\
			& \le 
			\left( 1 + \frac{| \Si\cap\Om|}{| \Si\cap B_{\de_{\Ga_0}(x_0)} (x_0) |} \right)^{\frac{1}{r}} \nr v - v_{\Si\cap\Om} \nr_{L^r ( \Si\cap \Om)}  ,
			\end{split}
		\end{equation*}
		where the inequality follows from Lemma \ref{lem:mediaor-mediauguale} (with $F:= \Si \cap B_{ \de_{\Ga_0}(x_0)} (x_0)$, $G:= \Si \cap \Om$, $p:=r$, and $\la := v_{\Si \cap \Om}$).
		The desired conclusion easily follows from \eqref{eq:BoasStraube-poincare} or \eqref{eq:John-Hurri-poincare} (both with $G:= \Si \cap \Om$). 
	\end{proof}
	

\begin{rem}
{\rm 
(i) In the particular case where $\Si=\RR^N$, we have $\Ga_0=\pa\Om$ and $\Ga_1=\varnothing$. Hence, for any $x_0\in \Om$, Theorem \ref{thm:mean value property harmonic in cones} returns the classical mean value property for harmonic functions in $\Om$. In turn, in the particular case $\Si=\RR^N$, Lemma \ref{lem:Poincare harmonic in cone} returns a result which is included in those discussed in \cite{Z, BS}.

(ii) Convexity of the cone is not necessary for the validity of Theorem \ref{thm:mean value property harmonic in cones} and Lemma \ref{lem:Poincare harmonic in cone}, provided that $v$ possesses enough regularity to allow the application of the divergence theorem in Theorem \ref{thm:mean value property harmonic in cones}.
}
\end{rem}

\subsection{Duality theorem relative to cones}\label{subsec:duality}
In the spirit of \cite{PS}, we now relate our new mean value formula with the mixed Serrin-type overdetermined problem considered in \cite{PT, Pog3}.

In what follows, in the notation of Setting \ref{Setting}, we assume that $\Si\cap\Om$ is such that the mixed boundary value problems
	\begin{equation}\label{eq:DirichletNeumann_problem harmonic}
	\begin{cases}
		\De v = 0 \quad & \text{ in } \Si\cap\Om,
		\\
		v= f(x) \quad & \text{ on } \Ga_0 , \quad \text{ where } f \in C^{2,\ga}_c (\Ga_0),
		\\
		v_\nu=0 \quad & \text{ on } \Ga_1 ,
	\end{cases}
\end{equation}
and 
\begin{equation}\label{eq:torsion mixed}
	\begin{cases}
		\De u = 0 \quad & \text{ in } \Si\cap\Om,
		\\
		u= 0 \quad & \text{ on } \Ga_0 ,
		\\
		u_\nu=0 \quad & \text{ on } \Ga_1 ,
	\end{cases}
\end{equation}
admit solutions belonging to $W^{2,2}(\Si\cap\Om) \cap W^{1,\infty}(\Si\cap\Om)$. The argument in \cite[Proposition 6.1]{PT} together with the results in \cite{AJ, Mazya} show that such an assumption is surely satisfied if $\Ga_0$ and $\Ga_1$ intersect orthogonally and $\Si$ is smooth outside the origin. 

\begin{thm}[Duality Theorem relative to cones]\label{thm:duality}
Let $\Si\cap\Om$ as above.

Then, the following two items are equivalent:
\begin{enumerate}[(i)]
\item The solution $u$ of \eqref{eq:torsion mixed} satisfies $u_\nu=const.$ on $\Ga_0$;
\item $\frac{1}{|\Si\cap\Om|}\int_{\Si\cap\Om} h \, dx =\frac{1}{|\Ga_0|}\int_{\Ga_0} h \, dS_x$ for any $h \in W^{2,2} (\Si\cap\Om) \cap W^{1,\infty}(\Si\cap\Om)$ satisfying~\eqref{eq:problem harmonic}.
\end{enumerate}
\end{thm}
\begin{proof}
Let $u$ be the solution of \eqref{eq:torsion mixed}.
For any $h\in W^{2,2}( \Si\cap\Om ) \cap W^{1,\infty}(\Si\cap\Om)$ satisfying \eqref{eq:problem harmonic}, we have that
$$
\int_\Om h \, dx = \frac{1}{N} \int_{\Om} (\De u) h \, dx = \frac{1}{N} \int_{\Ga_0 \cup \Ga_1} u_\nu \, h - u h_\nu \, dS_x= \frac{1}{N} \int_{\Ga_0} u_\nu \, h \, dS_x ,
$$
from which we obtain that
\begin{equation}\label{eq:proofduality}
\int_\Om h \, dx - \frac{| \Si\cap\Om |}{ |\Ga_0| } \int_{ \Ga_0 } h \, dS_x = \int_{\Ga_0} \left( \frac{u_\nu}{N} - \frac{| \Si\cap\Om |}{ |\Ga_0| } \right) \, h \, dS_x. 
\end{equation}

If (i) holds, the divergence theorem gives that $u_\nu=\frac{N | \Si\cap\Om |}{ |\Ga_0| }$ on $\Ga_0$, and (ii) easily follows from \eqref{eq:proofduality}.

On the other hand, if (ii) holds, (i) easily follows from \eqref{eq:proofduality}. In fact, for any $\varphi \in C^{2,\ga}_c(\Ga_0)$, we can choose $h$ solution of \eqref{eq:DirichletNeumann_problem harmonic} with $f= \varphi \left( \frac{u_\nu}{N} - \frac{| \Si\cap\Om |}{ |\Ga_0| } \right)$ on $\Ga_0$ to find that 
$$
\int_{\Ga_0} \left( \frac{u_\nu}{N} - \frac{| \Si\cap\Om |}{ |\Ga_0| } \right)^2 \, \varphi \, dS_x =0 ,
\quad \text{ for any } \varphi \in C^{2,\al}_c(\Ga_0),
$$
which gives (i).
\end{proof}

In the particular case where $\Si=\RR^N$, (i) of Theorem \ref{thm:duality} becomes the classical overdetermined Serrin problem (\cite{Se, We, NT}), and Theorem \ref{thm:duality} returns the classical duality result contained in \cite[Theorem I.1]{PS}. 
%
Partially overdetermined problems in the mixed boundary value setting (as (i) of Theorem \ref{thm:duality}) have been considered in \cite{PT,PT isoperimetric, Pog3} (see also \cite{CP, CL}). These are related (see the discussions in \cite{PT isoperimetric, Pog3}) to the classical isoperimetric inequality relative to convex cones, which was considered in \cite{LP, FI} (see also \cite{RR,CRS, DPV2, Indrei}).

In particular, \cite[Theorem 1.1]{PT} informs us that (i) of Theorem \ref{thm:duality} is equivalent to the following rigidity statement, provided that $\Si$ is a convex cone smooth outside the origin and $u \in W^{2,2} (\Si\cap\Om) \cap W^{1,\infty}(\Si\cap\Om)$:

\begin{itemize}
\item[$(iii)$] $\Si\cap\Om = \Si \cap B_R(z)$  and $u=\frac{1}{2}(|x-z|^2 -R^2 )$ for some $R>0$, and either $z$ is the origin or $z\in{\pa{\Si}}$ and $\Ga_0$ is a half sphere lying on a flat portion of $\pa\Si$.
\end{itemize}
See also \cite{Pog3} for related results
involving cones not necessarily smooth outside the origin.
In this case (see \cite[Theorem 2.10]{Pog3}), the equivalence between (i) of Theorem~\ref{thm:duality} and the rigidity statement $(iii)$ holds true by replacing $(iii)$ with
\begin{itemize}
	\item[$(iii)$] $\Si\cap\Om = \Si \cap B_R(z)$  and $u=\frac{1}{2}(|x-z|^2 -R^2 )$ for some $R>0$, and $z \in \left[ \mathrm{span}_{x \in \Ga_1 } \, \nu (x) \right]^{ \perp }$, where $\left[ \mathrm{span}_{x \in \Ga_1 } \, \nu (x) \right]^{ \perp }$ is the orthogonal complement in $\RR^N$ of the vector subspace $\mathrm{span}_{x \in \Ga_1} \, \nu (x) $.\footnote{In particular, 
		$
		\mathrm{span}_{x \in \Ga_1} \nu(x) = \RR^N
		$
		is a sufficient condition that guarantees that $z$ must be the origin. See \cite{Pog3} for details.
	}
\end{itemize}

In conclusion, we have equivalence between $(i), (ii), (iii)$. 

We refer the reader to \cite{Pog3,PPR} for applications of the Poincar\'e-type inequalities discussed in Subsection \ref{subsec:Poincare-type} to the study of quantitative stability estimates related to the symmetry result $(i) \implies (iii)$.
In the classical case $\Si = \RR^N$ (i.e., for the classical overdetermined Serrin problem),
quantitative stability estimates have been studied with different approaches by several authors, as already mentioned in Section \ref{sec:Quantitative rigidity delta}: see \cite{ABR, BNST, CMV, Feldman, MP2, MP3, MP6, GO, O, Scheuer} and references therein.

\subsection{Generalizations}
	We now discuss the extent to which the results in Section~\ref{sec:cones} can be extended to the anisotropic setting involving the so-called anisotropic Laplacian $\De_H$, where $H: \RR^N \to \RR$ is a norm. 
	
	Rigidity for the anisotropic Serrin overdetermined problem
	has been obtained in \cite{CianchiS, WX} for $\Si=\RR^N$ and \cite{Weng} for $\Si \subseteq \RR^N$.
	
	Theorem \ref{thm:mean value property harmonic in cones} can be extended to the anisotropic setting
	whenever the norm $H$ satisfies the additional assumption 
	\begin{equation}\label{eq:additional assumption from FK}
		\langle \na_\xi H (a) , \na_\xi H_0 (b) \rangle = \frac{ \langle a , b \rangle }{H(a) H_0 (b)} 
		\quad \text{ for } a, b \in \RR^N ,
	\end{equation}
	where $H_0$ is the dual norm, that is the polar function defined by
	$$
	H_0(x):= \sup_{\xi \neq 0} \frac{\langle x, \xi \rangle}{H(\xi)} \quad \text{ for } x \in \RR^N .
	$$
	We refer the reader to \cite{FK} for details on the setting, a proof
	in the case $\Si=\RR^N$, and a discussion on the necessity of \eqref{eq:additional assumption from FK} to recover the validity of the mean value property for $\De_H$-harmonic functions.
	
	We point out that \eqref{eq:additional assumption from FK} also allows an analogous extension for Theorem \ref{thm:duality}.
	
	We mention that \eqref{eq:additional assumption from FK} is equivalent to asking the norm $H$ to be of the form
	$H(\xi) = \sqrt{ \langle M \xi , \xi \rangle }$,
	for some symmetric and positive definite matrix $M \in \textrm{Mat}_N(\RR)$: we refer the interested reader to \cite[Theorem 1.2]{CFV}.



\section*{Acknowledgements}

The author is supported by the Australian Research Council (ARC) Discovery Early Career Researcher Award (DECRA) DE230100954 ``Partial Differential Equations: geometric aspects and applications'' and the 2023 J G Russell Award from the Australian Academy of Science. The author is
member of the Australian Mathematical Society (AustMS) and the Gruppo Nazionale Analisi Matematica
Probabilit\`{a} e Applicazioni (GNAMPA) of the Istituto Nazionale di Alta Matematica (INdAM).
	The author thanks Rolando Magnanini for bringing to his attention \cite{AM1,AM2} and the content of \cite[Theorem 1.2]{CFV}.

\end{document}